\documentclass[12pt]{amsart}
\usepackage{amssymb}
\usepackage{enumerate}
\newtheorem{Theorem}{Theorem}[section]
\newtheorem{Lemma}[Theorem]{Lemma}
\newtheorem{Proposition}[Theorem]{Proposition}

\def\Ree{\operatorname{Re}}
\def\Imm{\operatorname{Im}}
\def\Cal{\mathcal}

\def\DD{\mathbb D}

\def\CC{\mathbb C}

\def\dist{\operatorname{dist}}
\def\span{\operatorname{span}}

\def\RR{\mathbb R}

\def\Omega{\varOmega}
\let\phi=\varphi

\def\wdtl{\widetilde}
\def\C{\Bbb C}

\def\eps{\varepsilon}
\def\span{\mbox{span}}
\def\dist{\mbox{dist}}

\def\eps{\varepsilon}
\def\Delta{\varDelta}
\def\Phi{\varPhi}

\begin{document}
\title{On different extremal bases for $\CC$-convex domains}

\author{Nikolai Nikolov}
\address{Institute of Mathematics and Informatics\\ Bulgarian Academy of
Sciences\\1113 Sofia, Bulgaria}
\email{nik@math.bas.bg }

\author{Peter Pflug}
\address{Carl von Ossietzky Universit\"at Oldenburg, Institut f\"ur Mathematik,
Postfach 2503, D-26111 Oldenburg, Germany}
\email{peter.pflug@uni-oldenburg.de}

\author{Pascal J. Thomas}
\address{Universit\'e de Toulouse\\ UPS, INSA, UT1, UTM \\
Institut de Math\'e\-ma\-tiques de Toulouse\\
F-31062 Toulouse, France} \email{pthomas@math.univ-toulouse.fr}

\thanks{This note was written during the stay of the third named author at the Institute of
Mathematics and Informatics of the Bulgarian Academy of Sciences supported by a CNRS-BAS project
(September 2009), the stay of the first named author at the University of Oldenburg supported
by a DAAD grant (November 2009 - January 2010), and the stay of the first and second named authors
at the Erwin Schr{\"o}dinger Institute for Mathematical Physics in Vienna (November 2009, programme
``The $\bar\partial$-Neumann problem: analysis, geometry and potential theory''). The research was
partially supported by the DFG grant 436POL113/103/0-2. The authors thank H.~P.~Boas for his careful
reading and comments on a previous version of the note.}



\begin{abstract}
We discuss some extremal bases for $\CC$-convex domains.
\end{abstract}

\maketitle

\section{Introduction}

In order to estimate the Bergman kernel (on the diagonal) and the Carath\'eodory, Bergman, and Kobayashi metric on convex
(or linearly convex) domains special coordinates near the boundary were introduced by J.-H. Chen \cite{Che1989}, in his
Ph.~D. dissertation, and by J.~D. McNeal \cite{Neal1992,Neal1994}. A lot of further studies were based on this
orthonormal basis introduced in that papers, see e.g. \cite{MS1994,MS1997,Gau1997,Neal2001}.
We will call this basis a \textit{maximal basis}; a detailed construction, which also justifies
the name, will be described later. On the other side, for the same purpose a \textit{minimal basis} was
introduced, for example, in the papers \cite{Hef2002,Con2002,NP2003,Hef2004,DF2006,NPZ2009}. For a general notion
of extremal basis see \cite{CD2008}. Looking more carefully at the work based on the maximal basis it can be seen that from the very beginning a crucial property of that basis is used, but never proved. In this short note we will
present an example showing that exactly this property unfortunately does not hold in general for
the maximal basis. But it does hold for the minimal basis.  This property
may be phrased by saying that certain vectors connected with this basis are orthogonal to certain complex tangent planes;
details will be given later. Starting with this observation it should be asked whether the estimates for invariant metrics
given with help of the maximal basis still hold. In this note we show that the answer is positive.

Description of the two extremal bases for a domain $D\subset\CC^n$ containing no complex lines.

(a) \textit{Maximal basis}. Fix a point $q\in D$. Then choose a boundary point $p_1\in\partial D$
such that $m_1:=\|p_1-q\|=d_D(q)$, where $d_D(q)$ denotes the Euclidean boundary distance of
the point $q$. Put $a_1:=(p_1-q)/\|p_1-q\|$. Note that the point $p_1$ is not in general uniquely determined.
Denote by $H_1$ the affine hyperplane through $q$ which is orthogonal to the vector $a_1$, i.e.~ $H_1=q+\span\{a_1\}^\bot$.
Put $D_2:=D\cap H_1$. Choose a unit vector $a_2\in\span\{a_1\}^\bot$ and a boundary point $p_2\in\partial D\cap H_1$ such that $m_2:=\dist(q,a_2,\partial D)=\sup\{d_D(q;a)\}$, where the supremum is taken over all unit vectors $a$ in $\span\{a_1\}^\bot$, and $p_2=q+m_2a_2$. Here $d_D(q;a):=\sup\{r>0:q+r\DD a\subset D\}$ denotes the boundary distance of $q$ in direction of $a$ and $\DD$ is the open unit disc in $\CC$. In the next step put $H_2:=q+\span\{a_1,a_2\}^\bot$; $H_2$ is the affine $(n-2)$-dimensional plane through $q$ orthogonal to $\span\{a_1,a_2\}$.
Define $D_3:=H_2\cap D$ and continue this procedure, which finally leads to an orthonormal basis $a_1,\dots,a_n$, which is called a \textit{maximal  basis} of $D$ at $q$, to a sequence of positive numbers $m_2\geq\cdots\geq m_n$, and to boundary points $p_1,\dots,p_n$ with $p_j=q+m_ja_j$ for all $j$'s. Obviously, in this general context this basis depend on $q$ and it is, in general, not uniquely determined.

(b) \textit{Minimal basis}: Let $q$, $e_1:=a_1$, $s_1:=m_1$, $\wdtl p_1:=p_1$, where the $a_1$, $p_1$, and $m_1$ are taken from the former construction. Let $H_1$ be also as above. Define a new boundary point $\wdtl p_2=q+s_2e_2$, where $e_2\in\span\{e_1\}^\bot$ is an unit vector and
$s_2=\dist(q,\partial_{H_1}(D\cap H_1))$. Note that in this construction, opposite to the above one, $\wdtl p_2$ is chosen to be a nearest boundary point of $\partial D\cap H_1$ to $q$. Then put $H_2:=q+\span\{e_1, e_2\}^\bot$; $H_2$ is the affine $(n-2)$-dimensional plane through $q$ orthogonal to $\span\{e_1, e_2\}$. Define $D_3:=H_2\cap D$ and continue now this procedure for $H_3$ by always taking the nearest boundary point. This finally leads to an orthonormal basis $e_1,\dots,e_n$, which is called the \textit{minimal basis} of $D$ at $q$, to a sequence of positive numbers $s_1\leq s_2\leq\cdots\leq s_n$, and to boundary points $\wdtl p_1,\dots,\wdtl p_n$ of $D$ with $\wdtl p_j=q+s_je_j$, $1\leq j\leq n$. As above, this basis depend on $q$ and it is, in general, not uniquely determined.

Assume now that, in addition, $D$ is convex and $\Cal C^\infty$-smooth near a boundary point $p_1$
(of finite type). Let $r$ be its boundary function. Then the property indicated in the introduction can
be described as follows: Fix $q\in D$ on the inner normal at $p_1$, sufficiently near to $p_1$, and take the coordinate system given by the maximal  basis at $q$, i.e. take $q=0$ and write any point $z$ of $\CC^n$
as $z=\sum_{j=1}^n w_ja_j$. Then it is claimed (see, for example, \cite[Proposition 2.2 (ii)]{Che1989} and
\cite[Proposition 3.1 (i)]{Neal1992}) that
$$
\frac{\partial r(p_k)}{\partial w_j}=0,\quad j=k+1,\dots,n. \quad (*)
$$
In the original coordinate system this property reads as
$$
\sum_{s=1}^n\frac{\partial r(p_k)}{\partial z_s}a_{j,s} =0, \quad j=k+1,\dots,n.
$$
Therefore, an equivalent form to state (*) is to say that the vectors $a_j$, $j=k+1,\dots,n$, belong to the complex tangent space $T_{p_k}^\CC(\partial D)$ or that  $T_{p_k}^\CC(\partial D)\cap\span\{a_1,\dots,a_k\}^\bot=\span\{a_{k+1},\dots,a_n\}$. We should point out that exactly the property (*) is the basis of the arguments in those papers dealing with maximal bases (minimal bases have this crucial property).
But, as the following example will show, (*) is not true near the boundary of a domain in $\CC^3.$ Nevertheless, in section 3 it will be proved that the estimates obtained in terms of the maximal basis still hold.

\section{An Example}

Let $\beta_1$ and $\beta_2$ be real numbers with $0<\beta_2<\beta_1<1$. Define
$$
D:=\{z\in\CC^2\times\CC:\rho(z)+|z_3|^2<1\},
$$
where $\rho(z)=x_1^2+\beta_1y_1^2+x_2^2+\beta_2y_2^2.$ Note that $D$ is a strictly
(pseudo)con\-vex domain with real-analytic boundary. Fix $q=(0,0,\delta)$, $0<\delta<1.$
Then following the construction of the maximal basis of $D$ at $q$ leads to
$m_1=1-\delta$ and $a_1=p_1=(0,0,1)$. In the next step the construction gives the domain
$$
D_\delta:=\{z\in\CC^2:\rho(z)<1-\delta^2\}.
$$
Note that $D_\delta$ is up to a dilatation $D_0$. So it suffices to study $D_0.$ If we have a maximal basis, write $a_2= (b,0)$ where
$b \in \CC^2$. Then $\span\{a_1,a_2\}^\bot$ is spanned by
$(-\overline b_2, \overline b_1, 0)$.  So
put
$$\Cal T:=\{b\in\CC^2:\frac{\partial \rho(b)}{\partial z_1}(-\overline b_2)+
\frac{\partial \rho(b)}{\partial z_2}(\overline b_1)=0\}.$$

\begin{Lemma} \label{lemma}
$\Cal T=\{b\in\CC^2: b_1=0 \text{ or } b_2=0 \text{ or } \Imm b_1=\Imm b_2=0\}$.
\end{Lemma}

\begin{proof}
Simple calculations show that $b\in\Cal T$ if and only if
$$
\begin{aligned}(\beta_1-\beta_2)\Imm b_1\Imm b_2&=0\\
(1-\beta_1)\Imm b_1\Ree b_2&=(1-\beta_2)\Imm b_2\Ree b_1,
\end{aligned}
$$
from which the statement of the lemma follows.
\end{proof}

The next result shows that the property (*) is, in general, not true for a maximal basis.
Let  $p_2\in\partial D_0$ be such that
$$\frac{d_{D_0}(0;p_2)}{\|p_2\|}=m_2=\sup_{a\in\CC^2,\|a\|=1}d_{D_0}(0;a).$$

\begin{Proposition}$p_2\not\in\Cal T.$
\end{Proposition}

\begin{proof}
Let $b\in\Cal T$ be a unit vector. Observe that $\rho(re^{i\alpha}b)<1$ for all $\alpha\in\RR$ if and only if
$r^2R(b)<1,$ where $R(b):=\max\{\rho(e^{i\alpha}b):\alpha\in\RR\}$. Therefore, $d_{D_0}(0;b)=1/\sqrt{R(b)}$.
Write $b=(e^{i\phi_1}\cos\Theta,e^{i\phi_2}\sin\Theta)$, where $0\leq\Theta<2\pi$ and $0\le\phi_1,\phi_2\le\pi/2.$
According to Lemma \ref{lemma}, there are three possibilities for $b$:

$\bullet$ $\Theta=0$ or $\Theta=\pi$: $\rho(e^{i\alpha}b)=\cos^2(\alpha+\phi_1)+\beta_1\sin^2(\alpha+\phi_1).$

$\bullet$ $\Theta=\pi/2$ or $\Theta=3\pi/2$: $\rho(e^{i\alpha}b)=\cos^2(\alpha+\phi_2)+\beta_2\sin^2(\alpha+\phi_2).$

$\bullet$ $\phi_1=\phi_2=0$: $\rho(e^{i\alpha}b)=\cos^2\alpha+\sin^2\alpha(\beta_1\cos^2\Theta+\beta_2\sin^2\Theta).$

\noindent Hence $R(b)=1$ in all the three cases.

On the other hand, there is a unit vector $b^\ast\in\Bbb C^2$ with $R(b^\ast)<1$ which implies that $p_2\not\in\Cal T$.
To define $b^\ast$,
take $\Theta:=\pi/4$ $\phi_1:=0$ and $\phi_2:=\pi/2.$ Then
$2\rho(e^{i\alpha}b^\ast)=1+\beta_2+(\beta_1-\beta_2)\sin^2\alpha.$
Since $\beta_1<\beta_2<1$, it follows that  $R(b^\ast)=\frac{1+\beta_2}{2}<1.$
\end{proof}

\section{Estimates and localization}

Let $D\subset\CC^n$ be a $\C$-convex domain containing no complex lines, i.e.
any non-empty intersection with a complex line is biholomorphic to $\DD$
(cf. \cite{APS2004,Hor1994}). For $z\in D$ denote by $e_1(z),\dots,e_n(z)$
a minimal basis at $z$ and by $a_1(z),\dots,a_n(z)$ a reordered maximal basis at $z$ which
means that the new $a_1(z)$ is the old one, but then $a_2(z)=a_n$, $a_3(z)=a_{n-1}$, etc.
Let $s_1(z)\le\dots\le s_n(z)$ and $m_1(z)\le\dots\le m_n(z)$ be the respective numbers
(recall that $s_1(z)=m_1(z)=d_D(z)$).
Set $s_D(z):=\prod_{j=1}^ns_j(z)$ and $m_D(z):=\prod_{j=1}^nm_j(z).$
Moreover, denote by $K_D(z)$ and $F_D(z;X)$ the Bergman kernel and any of the invariant metrics of $D$, respectively.
For $X\in\C^n,$ set
$$E_D(z;X):=\sum_{j=1}^n\frac{|\langle X,e_j(z)\rangle|}{s_j(z)},\quad
A_D(z;X):=\sum_{j=1}^n\frac{|\langle X,a_j(z)\rangle|}{m_j(z)}.$$
We shall write $f(z)\lesssim g(z)$ if $f(z)\le c g(z)$ for some constant $c>0$ depending only on
$n;$ $f(z)\sim g(z)$ means that $f(z)\lesssim g(z)\lesssim f(z).$ By \cite{NPZ2009}, we know that
$$K_D(z)\sim 1/s^2_D(z),\quad F_D(z;X)\sim E_D(z;X)\sim 1/d_D(z;X)$$
(for weaker versions of these results, see \cite{NP2003,Blu2005,Lie2005}).
For short, sometimes we shall omit the arguments $z$ and $X$. It follows by \cite[Lemma 15]{NPZ2009}
that $$K_D\lesssim 1/m^2_D,\quad F_D\lesssim A_D.$$ In particular,
$$1/d_D(z;X)\sim E_D(z;X)\lesssim A_D(z;X)$$
The main consequence of the (wrong) property $(*)$ for maximal bases of a smooth convex bounded domain of
finite type is the fact that
$$A_D(z;X)\sim_D 1/d_D(z;X),$$ where the constant in $\sim_D$ depends on $D.$
Using this fact, it is shown in \cite{Che1989,Neal1994,Neal2001} that
$$K_D\sim_D 1/m_D^2,\quad F_D\sim_D A_D$$
The following two propositions imply that fortunately these estimates still hold.

The first one is contained in \cite{Hef2004} for the case of a smooth convex bounded
domain of finite type. The proof there invokes the estimate $1/d_D(z;X)\sim_D A_D(z;X)$
but, in fact, it uses only the trivial part of this estimate: $1/d_D(z;X)\lesssim_D A_D(z;X).$

\begin{Proposition}\label{pr1} Let $D\subset\C^n$ be a $\C$-convex domain containing no
complex lines. Then $m_j(z)\sim s_j(z)$, $j=1,\dots,n$, $z\in D$.
\end{Proposition}

\begin{proof} Fix $z\in D$ and put $m_j=m_j(z),$ $s_j=s_j(z)$ First, we shall prove that $m_j\lesssim s_j.$
Since $E_D\lesssim A_D,$ it is enough show that if $E_D\le cA_D,$ then $m_j\le c' s_j,$ where $c'=n!c.$

Expanding the determinant of the matrix of the unitary transformation between the bases, it follows that
$\prod_{j=1}^n|\langle a_j,e_{\sigma(j)}\rangle|\ge 1/n!$ for some permutation $\sigma$ of $\{1,\dots,n\}$
In particular,
$|\langle a_j,e_{\sigma(j)}\rangle|\ge 1/n!.$ Then $E_D(z;a_j)\le c A(z;a_j)$
implies that $m_j\le c's_{\sigma(j)}.$

Assume now that $c's_k<m_k$ for some $k.$ Then
$$c's_k<m_k\le m_j\le c's_{\sigma(j)},\quad j\ge k.$$
This shows that $\sigma(j)>k$
for any $j\ge k,$ which is a contradiction, since $\sigma$ is a permutation.

The above arguments show that $\wdtl s_j\sim s_j,$ where $\wdtl s_j$ are the respective
numbers for another minimal basis at $z.$ So we may assume that $e_1=a_1.$ We know that
$m_1=p_1.$ It remains to prove that $m_k\gtrsim s_k$ for $k\ge 2.$
Choose a unit vector $a_k'$ in $\span(e_k,\dots,e_n)$ orthogonal to $a_{k+1},\dots a_n$
($a_n'=e_n$ if $k=n$). Then $a_k'$ is also orthogonal to $a_1=e_1.$ Hence
$m_k\ge d_D(z;a_k')$ (by construction of a maximal basis).
On the other hand, since $a_k'$ is orthogonal to $e_1,\dots,e_{k-1},$ then
$$\frac{1}{d_D(z;a_k')}\sim E_D(z;a_k')=
\sum_{j=k}^n\frac{|\langle a_k',e_j\rangle|}{s_j}\lesssim\frac{1}{s_k}.$$
So $m_k\ge d_D(z;a_k')\gtrsim s_k.$
\end{proof}

\begin{Proposition}\label{pr2} Let $D$ be as in Proposition \ref{pr1}. Then $A_D\sim E_D.$
\end{Proposition}

\begin{proof} Using the inequality $E_D\lesssim A_D$ and Proposition \ref{pr1},
it is enough to show that for any $k,$
$$\frac{|\langle X,a_k\rangle|}{s_k}\lesssim E_D (z;X).$$

Set $b_{jk}=\langle a_j,e_k\rangle.$
Since $$\frac{1}{s_j}\sim\frac{1}{d_D(z;a_j)}\sim E_D(z;a_j)\ge\frac{|b_{jk}|}{s_k},$$
it follows that $|b_{jk}|\lesssim s_k/s_j.$
The unitary matrix $B=(b_{jk})$
transforms the basis $e_1,\dots,e_n$ to the basis $a_1,\dots,a_n$. For the inverse matrix $C=(c_{jk})$ we have
$$
|c_{jk}|\le\sum_{\sigma}|b_{1\sigma(1)}\dots b_{k-1,\sigma(k-1)}b_{k+1,\sigma(k+1)}\dots b_{n,\sigma(n)}|
$$
$$
\lesssim\sum_{\sigma}\frac{s_{\sigma(1)}}{s_1}\dots\frac{s_{\sigma(k-1)}}{s_{k-1}}
\frac{s_{\sigma(k+1)}}{s_{k+1}}\dots\frac{s_{\sigma(n)}}{s_n}=\sum_{\sigma}\frac{s_k}{s_j}=(n-1)!\frac{s_k}{s_j},
$$
where $\sigma$ runs over all bijections from $\{ 1,\dots, k-1, k+1, \dots, n\}$
to $\{ 1,\dots, j-1, j+1, \dots, n\}$.

It follows that
$$
\frac{|\langle X,a_k\rangle|}{s_k}\le
\sum_{j=1}^n|\langle X,e_j\rangle|\frac{|b_{kj}|}{s_k}
=\sum_{j=1}^n|\langle X,e_j\rangle|\frac{|\overline c_{jk}|}{s_k}\lesssim E_D.$$
\end{proof}

\noindent{\bf Remark.} Replace the construction of a maximal basis by the following: choose ``minimal'' discs on steps
$1,\dots,k$ and ``maximal'' discs on steps $k+1,\dots, n-1$ (the $n$-th choice is unique); $k=n-1$ provides a
minimal basis, $k=1$ -- a maximal basis, and $k=0$ -- a basis with no ``minimal'' discs. Note that Propositions \ref{pr1}
and \ref{pr2} remain true with $A_D$ expressed in the new basis. (This construction has an obvious real analog).
\smallskip

Let $a$ be a boundary point of a bounded domains $D\subset\C^n.$ It is easy to see that for any neighborhood $U$
of $a$ one has that $s_D\sim_\ast s_{D\cap U},$
and $E_D\sim_\ast E_{D\cap U}$ near $a,$ where the constant in $\sim_\ast$ depends on $D$ and $U$
(the same holds for $m_D$ and $A_D$). Assume that $D$ is $\Cal C^2$-smooth
and (weakly) linearly convex near $a$ (cf. \cite{APS2004,Hor1994} for this and other notions of convexity). Then
Proposition \ref{pr3} below and the localization principle for the Kobayashi metric $\kappa_D$
(cf.~\cite{JP1993}) imply that $\kappa_D\sim_D E_D$ near $a$ (the constant in $\sim_D$ depends on $D$).
If, in addition, $D$ is pseudoconvex, then the same principle for $K_D$ and the
Bergman metric $b_D$ (cf.~\cite{JP1993}) implies that $K_D\sim_D 1/s^2_D$ and $b_D\sim_D E_D$ near $a.$
Assume that $a$ is a $\Cal C^\infty$-smooth finite type point (but $D$ not necessarily bounded). Then $a$ is a
local holomorphic peak point (see \cite{DF2003}), and strong localization principles (cf. \cite{Nik2002})
imply that $\kappa_D\sim E_D(=E_{D\cap U})$ and (if $D$ is pseudoconvex)
$K_D\sim 1/s^2_D(=1/s^2_{D\cap U}),$ $b_D\sim E_D$ near $a.$

\begin{Proposition}\label{pr3} Let $a$ be a $\Cal C^k$-smooth boundary point ($2\le k\le\infty$) of a domain
$D\subset\C^n$ with the following property: for any $b\in\partial D$ near $a$ there are a neighborhood $U_b$
such that $D\cap U_b\cap T^\CC_b(\partial D)=\varnothing.$ Then there is a $\Cal C^k$-smooth $\Bbb C$-convex
domain $G\subset D$ and a neighborhood $U$ of $a$ such that $D\cap U=G\cap U$.
\end{Proposition}

\begin{proof} We may assume that $a=0.$ Denote by $H_f(z;X)$ the Hessian of a $C^2$-smooth function $f$.
Set $B_s:=\Bbb B_n(0,s)$ ($s>0$) and
$$r(z):=\left\{\begin{array}{ll}
-d_D(z),&z\in D\\
d_D(z),&z\not\in D.\end{array}\right.$$
It follows by the differential inequality for $r^2$ in the proof of
\cite[Proposition 2.5.18 (ii)$\Rightarrow$(iii)]{APS2004} that there is an $\eps>0$ such
$r$ is a $\Cal C^k$-smooth defining function of $D$ in $B_{3\eps}$ and $H_r(z;X)\ge 0$
if $\langle\partial r(z),\overline X\rangle=0$ and $z\in D\cap B_{2\eps}.$
Then the proof of \cite[Lemma 1]{DF1977} implies that there is a $c>0$ such that
$H_r(z;X)\ge-c|X|\cdot|\langle\partial r(z),\overline X\rangle|,$ $z\in D\cap B_{2\eps}.$ We may assume
that $2\eps c\le 1$ and $D\cap B_\eps$ is connected. Choose now a smooth function $\chi$ such that
$\chi(x)=0$ if $x\le\eps^2$ and
$\chi'(x),\chi''(x)>0$ if $x>\eps^2.$ Set $\theta(z)=\chi(|z|^2).$ We may find a $C\ge 1/2$ such that
$$B_{2\eps}\Supset G':=\{z\in B_{2\eps}:0>\rho(z)=r(z)+C\theta(z)\}\subset D.$$

Further, the inequalities $2c\eps\le1$ and $|\langle\partial\theta(z),\overline X\rangle|\le\chi'(|z|^2)|z|\cdot|X|$ give
$\chi'(|z|^2)|X|>c|\langle\partial\theta(z),\overline X)|$ if $z\in B_{2\eps}\setminus\overline{B_\eps}$ and $X\neq 0.$
This together with
$$H_r(z;X)\ge-c|X|\cdot|\langle\partial r(z),\overline X\rangle|,\quad z\in G',$$
$$
H_{\rho}(z;X)=H_r(z;X)+4C\chi''(|z|^2)
\mbox{Re}^2\langle z,X\rangle+2C\chi'(|z|^2)|X|^2,
$$
$2C\ge 1,$ and the triangle inequality show that
$$H_{\rho}(z;X)\ge-c|X|\cdot|\langle\partial\rho(z),\overline X\rangle|,\quad z\in\overline{G'}.$$
Moreover, the last inequality is strict if $z\in\overline{G'}\setminus\overline{B_\eps}$ and $X\neq 0.$
This implies that $\partial\rho\neq 0$ on $\partial G'\setminus\overline{B_\eps};$
(otherwise, $\rho$ will attain local minima at some point of this set which
is impossible). So $\partial\rho\neq 0$ on $\partial G'.$

Let $G$ be the connected component of $G'$ containing $D\cap B_\eps.$ Then
\cite[Proposition 2.5.18]{APS2004} (see also \cite[Proposition 4.6.4]{Hor1994})
implies that $G$ is a $\Cal C^k$-smooth $\Bbb C$-convex domain.
\end{proof}

\def\bibname{References}
\bibliographystyle{amsplain}

\end{document}